\documentclass[12pt,leqno]{article}
\usepackage{graphicx, fullpage}
\usepackage{amsmath,amssymb,amsthm,amscd, bm}
\usepackage{fancyhdr}
\usepackage[mathscr]{eucal}
\usepackage{amsfonts}
\begin{document}
\parskip=6pt

\theoremstyle{plain}

\newtheorem {thm}{Theorem}[section]
\newtheorem {lem}[thm]{Lemma}
\newtheorem {cor}[thm]{Corollary}
\newtheorem {defn}[thm]{Definition}
\newtheorem {prop}[thm]{Proposition}
\numberwithin{equation}{section}
\def\cal{\mathcal}
\newcommand{\cF}{\cal F}
\newcommand{\cA}{\cal A}
\newcommand{\cC}{\cal C}
\newcommand{\cO}{{\cal O}}
\newcommand{\cE}{{\cal E}}
\newcommand{\cU}{{\cal U}}
\newcommand{\cM}{{\cal M}}
\newcommand{\cD}{{\cal D}}
\newcommand{\cK}{{\cal K}}
\newcommand{\bC}{\mathbb C}
\newcommand{\bP}{\mathbb P}
\newcommand{\bN}{\mathbb N}
\newcommand{\bA}{\mathbb A}
\newcommand{\bR}{\mathbb R}
\newcommand{\var}{\varepsilon}
\newcommand{\End}{\text{End }}
\newcommand{\loc}{\text{loc}}
\newcommand{\dist}{\text{dist}}
\newcommand{\rk}{\roman{rk }}
\renewcommand\qed{ }
\begin{titlepage}
\title{\bf Extrapolation, a technique to estimate\thanks{Research in part supported by NSF grant DMS--1464150}}
\author{L\'aszl\'o Lempert\\ Department of  Mathematics\\
Purdue University\\West Lafayette, IN
47907-2067, USA}
\thispagestyle{empty}
\end{titlepage}
\date{}
\maketitle
\abstract
We introduce a technique to estimate a linear operator by embedding it in a family $A_t$ of operators, $t\in(\sigma_0,\infty)$, with suitable curvature properties. One can then estimate the norm of each $A_t$ by bounds that hold in the limit $t\to\sigma_0$, respectively, $t\to\infty$. We illustrate this technique on an extension problem that arises in complex geometry.
\endabstract

\section{Introduction}

This paper grew out of a joint paper with Berndtsson, that used Berndtsson's theorem on the curvature of direct images to prove an Ohsawa--Takegoshi type extension result for holomorphic functions, \cite{BL, B, OT}. Here we distill from \cite{BL} an abstract theorem on estimating Hilbert and Banach space operators by a technique we call extrapolation,  and discuss a dual theorem of similar nature as well as an application.

The related notion of interpolation is a well established method in harmonic analysis to estimate Banach space operators. Loosely stated, given two pairs of Banach spaces and operators
$$
A_0: E_0\to F_0, \qquad A_1:E_1\to F_1
$$
among them, one can construct interpolating families $E_t, F_t$ of Banach spaces and operators
 $A_t: E_t\to F_t$,  $t\in(0, 1)$, whose norms can be bounded in terms of the norms of $A_0$ and $A_1$. There are several inequivalent ways to construct the interpolation spaces. In the complex method, going back to M. Riesz and Thorin, \cite{R, T}, the estimate on $\|A_t\|$ is obtained from Hadamard's Three Circles Theorem, i.e., from properties of subharmonic functions. 

In extrapolation---again loosely stated---we will be given only one operator, through which we will estimate an entire family of operators $A_t:E_t\to F_t$. This again depends on subharmonicity and convexity, derived from curvature properties of the bundles that the $E_t, F_t$ form.

To be concrete, consider two hermitian holomorphic vector bundles $(E, h)$ and $(F,k)$ over the same base $S$, and a holomorphic homomorphism $A:E\to F$. We assume the metrics $h, k$ are of class $C^2$. The fibers of $E$ and $F$ could be Hilbert spaces, but finite rank bundles will already illustrate the main idea. By the norm $\|A\|: S\to [0,\infty)$ of $A$ we mean the function $\|A\|(s)=$ operator norm of $A_s$, where $A_s=A|E_s:(E_s, h)\to (F_s, k)$. In \cite{L} we have defined what it means for $A$ to decrease curvature; this definition will be reproduced in section 2 and Theorem 3.1.

\begin{thm} (i) If $A:E\to F$ decreases curvature, then $\log\|A\|$ is plurisubharmonic. 

(ii) If in addition $S$ is a half plane
 $$S=\{s\in\bC: \text{Re }s>\sigma_0\},$$
$\|A\|$ is bounded, and $\|A\|(s)=\|A\|(\text{Re }s)$ for $s\in S$, then $\|A\|(t)$ is a decreasing function of $t\in(\sigma_0,\infty)$.
\end{thm}

So, in the setting of (ii), if we can estimate $\|A\|(t)$ for some $t=t_0$, the same estimate will hold for $t\ge t_0$, an instance of extrapolation. For this conclusion we had to know in advance that $\sup_S\|A\|<\infty$. Thus (ii) has the flavor of an a priori estimate: assuming a crude bound on all $\|A\|(t)$, a sharper bound follows.

For the application we have in mind we will need an analogous result but with a homomorphism $A$ that increases curvature. We could formulate a clean theorem for quite special bundles $E, F$. First off, both bundles will be trivial, $E=F=S\times V\to S$. In this case their metrics  $h, k$, can be thought of as a families $h_s, k_s$ of, say, Hilbertian norms on $V$. While from the theorem below $(F,k)$ has disappeared, it is really about the norm of the identity homomorphism $A:(E,k)\to(E, h)$, when $k_s$ is independent of $s$ (and so $(E,k)=(F,k)$ is flat). 

\begin{thm} Suppose that $S=\{s\in\bC: Res>\sigma_0\}$ and $E= S\times V\to S$ is a trivial Hilbert bundle, with $\|\ \ \|$ the norm on $V$. Suppose further that a $C^2$ metric $h$ on $E$ has semipositive curvature, $h_s=h_{Res}$, and with some $c>0$
$$\inf_{s\in S} h_s(v)\ge c\|v\|\qquad \text{ for all } v\in V.$$
Then $h_t(v)$ is an increasing function of $t\in(\sigma_0, \infty)$ for any $v\in V$, in particular, $h_t(v)\le\lim\limits_{\tau\to\infty} h_{\tau}(v)$.
\end{thm}

In section 3 we will formulate and prove results more general than Theorems 1.1 and 1.2, that apply to metrics and norms that are not necessarily Hilbertian, see Theorems 3.1, 3.2.

Various problems in analysis of course boil down to estimating a norm on a space $V$. The way to do this by extrapolation is to include the norm in a family $h_s$ of norms, $s\in S$, that satisfies the assumptions of say, Theorem 1.2, and estimate $h_s$ in the limit $\text{Re }s\to\infty$. Whatever estimate one can prove in the limit will then hold for the norm we started with. The success of this approach depends on whether $h_s$ gets easier to estimate as $\text{Re }s\to \infty$; in the application in section 4 this will be so.

The method is reminiscent of the continuity method. In the continuity method one attacks a problem $P$  first by connecting it, through a family $P_t$, $t\in[0,1], P_0=P$, of problems with a problem ${P}_1$ that one knows how to solve, and then by studying the parameter values $t$ for which ${P}_t$ can be solved. In this approach the choice of ${P}_t$ connecting $P$ and $P_1$ is largely arbitrary, and the success of the method depends on proving suitable estimates for the solutions of $P_t$. By contrast, it is estimates that extrapolation furnishes. One still needs to connect problem $P$ through a family $P_t$ with a limiting problem $\lim_{t\to\infty} P_t$ one knows how to solve, but the family $P_t$ should have suitable curvature properties and in its choice is quite restricted.

In harmonic analysis the notion of extrapolation is not new. Its first appearance is in Yano's paper \cite{Y} from 1951 that studies what certain $L^p\to L^p$ estimates of an operator for all $p>1$ imply in the limit $p\to 1$. In the 1980s Rubio de Francia introduced a different notion of extrapolation of operators and estimates, see \cite {Ru1,Ru2}. Knowing that a certain operator is bounded in weighted $L^p$ spaces, for fixed $p\in[1, \infty)$ but simultaneously for a large class of weights, he concludes that the operator is bounded in correspondingly weighted $L^q$ spaces for all $q\in[1,\infty)$.  Originally motivated by Yano's theorem, in the 1990s Jawerth and Milman proposed a general theory of extrapolation. As Mart\'\i n and Milman subsequently pointed out, even Rubio de Francia's results fit in this framework, see \cite{JM, MM}. However, apart from formal similarity, that theory of extrapolation seems to be unrelated to the results in this paper.

I am grateful to Richard Rochberg who pointed me to relevant literature, and suggested that the paper be included in this volume in memory of Bj\"orn Jawerth.

\section{Background}
A holomorphic Banach bundle is a holomorphic map $\pi: E\to S$ of complex Banach manifolds, with each fiber $E_s=\pi^{-1}\{s\}$ endowed with the structure of a complex vector space. It is required that for each $s_0\in S$ there be a neighborhood $U\subset S$, a complex Banach space $W$, and a biholomorphic map $\Phi:U\times W\to\pi^{-1}U$ (a local trivialization) that maps $\{s\}\times W$ to $E_s$ linearly, $s\in U$. If $W$ can be chosen a Hilbert space, we speak of a holomorphic Hilbert bundle.

A metric on $E$ is a locally uniformly continuous function $p:E\to [0,\infty)$ that restricts on each fiber $E_s$ to a norm $p_s$ inducing the topology of the fiber. We measure the curvature of a metric as follows, see also \cite{L}. First, if $D\subset\bC$ is open, $z_0\in D$, and $u:D\to\bR$ is upper semicontinuous, we let 
\begin{equation}
\Lambda u(z_0)=\limsup_{r\to 0}\, r^{-2}\Big(\int^1_0 u(z_0+re^{2\pi i\theta}) d\theta-u(z_0)\Big)\in [-\infty,\infty].
\end{equation}
 In particular, $\Lambda u=\partial^2 u/\partial z\partial\overline z$ when $u\in C^2(D)$. Second, if $S$ is a complex manifold, $u:S\to\bR$ is upper semicontinuous, and $\xi\in T^{1,0}S$, we let
$$\xi\overline\xi u=\inf \Lambda(u \circ f)(0)\in [-\infty, \infty],$$
where the $\inf$ is taken over holomorphic maps $f$ of some neighborhood of $0\in\bC$ into $S$ that send $\partial/\partial z\in T_0^{1,0}\bC$ to $\xi$. 

Third, returning to a holomorphic Banach bundle $E\to S$ endowed with a metric $p$, we define the Kobayashi curvature $K_\xi(v)$ of $p$, for $\xi\in T^{1,0}_s S$, $v\in E_s\setminus \{0\}$, by
\begin{equation}
K_\xi(v)=-\inf\, \xi\overline\xi\log p(\varphi)(s),
\end{equation}
the $\inf$ taken over sections $\varphi$ of $E$, holomorphic near $s$, such that $\varphi(s)=v$. For example, by \cite[ Lemma 2.3]{L} $K_\xi(v)\le 0$ for all $\xi, v$ if and only if $\log p(\varphi)$ is plurisubharmonic for all local holomorphic sections of $E$. (The above definition (2.2) of curvature differs from [L, (2.6)] by a factor of 2.) If $E$ is a Hilbert bundle, $p$ is the metric associated with a hermitian metric $h$ of class $C^2$, and $R$ is the curvature operator of $h$ (so it is an End $E$ valued $(1,1)$ form on $S$), then
$$K_\xi(v)=\frac{h(v, R(\xi,\overline\xi)v)}{ 2 h(v,v)}.$$

The duals $E^*_s$ of the fibers $E_s$ of a holomorphic Banach bundle also form a holomorphic Banach bundle, denoted $E^*$. If $E$ was locally isomorphic to $U\times W$, then $E^*$ will be locally isomorphic to $U\times W^*$. If $E$ is endowed with a metric $p$, the norms $p^*_s$ on $E^*_s$, dual to the norms $p_s$, form a metric on $E^*$, denoted $p^*$. Let $K^*$ stand for the Kobayashi curvature of $p^*$.

From now on we assume $\dim S<\infty$.

\begin{lem}
If the Kobayashi curvature of $(E^*, p^*)$ is semipositive $(K^*\ge 0)$, then the Kobayashi curvature of $(E,p)$ is seminegative $(K\le 0)$.
\end{lem}

When $p$ is a Hilbertian metric, the converse is also true, but we do not know whether this converse holds in general. We also do not know whether in general $K^*\le 0$ is equivalent to $K\ge 0$.---For the proof of the lemma  we need a characterization of plurisubharmonicity:

\begin{lem}
Suppose $w:S\to\bR$ is upper semicontinuous, and for every $\xi\in T^{1,0} S$ there is a holomorphic map $f$ of a neighborhood of $0\in\bC$ into $S$, sending $\partial/\partial z\in T^{1,0}_0\bC$ to $\xi$, such that
\begin{equation}
\Lambda (w\circ f)(0)\ge 0.
\end{equation}
Then $w$ is plurisubharmonic.
\end{lem}

In \cite{L} we already used a similar result. In Lemma 2.3 there the assumption was that (2.3) holds for all holomorphic maps $f$, which allowed us to quote a theorem of Saks, \cite{Sa}. It turns out that Saks's proof, through a maximum principle, can be tweaked to prove the stronger Lemma 2.2.

\begin{proof}
We can assume that $S$ is an open subset of some $\bC^m$. It follows from (2.1) that if $w\in C^2(S)$ then $\Lambda(w\circ f)(0)=\xi\overline\xi w=\partial\overline\partial w(\xi,\overline\xi)$, and the $\limsup$ there is a limit. Of course, the point of Lemma 2.2 is that $w$ is just assumed upper semicontinuous. However, upon replacing $w(s)$ by $w(s)+\varepsilon|s|^2$, with $\varepsilon >0$, we can assume without loss of generality that whenever $\xi\in T^{1,0}S$ is nonzero, in (2.3) the strict inequality
$$
\Lambda(w\circ f)(0)>0
$$
holds. Consider now a $u\in C^2(S)$ such that for every $s\in S$ there is a nonzero $\xi\in T^{1,0}_s S$ with $\xi\overline\xi u=\partial\overline\partial u(\xi,\overline\xi)=0$. If $f$ is chosen as in the lemma, then $\Lambda(w\circ f - u\circ f)>0$, which shows that $w-u$ cannot have a local maximum in $S$.

To prove that $w$ is plurisubharmonic, we need to take a closed disc $\Delta \subset S$, which we assume to be 
$$\Delta=\{s\in\bC^m: |s_1|\le r,\quad s_2=\dots=s_m=0\};$$
a function $h\in C(\Delta)$, harmonic on the open disc, such that $u<h$ on the boundary of $\Delta$; and prove that this implies $u<h$ on all of $\Delta$. Choose $\delta>0$ so that
$$P=\{s\in\bC^m: |s_1|\le r,\quad |s_2|,\dots , |s_m|\le\delta\}\subset S,$$
and define
 $$u(s)=h(s_1,0,\ldots,0)+c(|s_2|^2+\cdots+|s_m|^2),\qquad s\in P,$$
 where $c>0$ is so large that $w<u$ on $\partial P$. Then $\xi\overline\xi u=0$ when $\xi=\partial / \partial s_1$. By what we have said above, this implies $w-u$ has no local maximum in int $P$, and so it is everywhere negative. In particular, $w<u=h$ on $\Delta$, as needed.
\end{proof}

\begin{proof}[ Proof of Lemma 2.1] 
We need to show that $\log p(\varphi)$ is plurisubharmonic for any local holomorphic section $\varphi$ of $E$. Again we can assume $S$ is an open subset of some $\bC^m$, and upon replacing $p$ by $pe^{\varepsilon|s|^2}$, $\varepsilon >0$, that $K^*_\xi(l)>0$ when $\xi\ne 0, l\ne 0$. Fix $s\in S$ and assume $\varphi(s)=v\ne 0$. By the Banach--Hahn  theorem there is a linear form $l\in E^*_{s}$ of norm 1 such that $\langle l, v\rangle=p(v)$. Let $\xi\in T^{1,0}_{s} S \setminus \{0\}$. As $K^*_\xi(l)>0$, there are a section $\psi$ of $E^*$, holomorphic near $s$, such that $\psi(s)=l$, and a holomorphic map $f$ of some neighborhood of $0\in\bC$ into $S$, sending $\partial / \partial z\in T^{1,0}_0\bC$ to $\xi$, these two satisfying
\begin{equation}
\Lambda \log p^*(\psi\circ f)(0)\le 0.
\end{equation}
At the same time
$$\log |\langle \psi\circ f, \varphi\circ f\rangle|\le \log p^*(\psi\circ f)+\log p(\varphi\circ f),$$
with equality at $0\in\bC$. Since the left hand side is harmonic near 0, by (2.4)
$$\Lambda \log p(\varphi\circ f)(0)\ge 0.$$
Applying Lemma 2.2 with $w=\log p(\varphi)$ we obtain that $\log p(\varphi)$ is plurisubharmonic wherever it is finite; and this implies it is plurisubharmonic in fact everywhere.
\end{proof}

We will need one more lemma, familiar in the context of hermitian vector bundles.

\begin{lem}
Suppose $F\to S$ is a holomorphic Banach bundle with reflexive fibers and $q$ is a metric on $F$. Letting $E\to S$ be another holomorphic Banach bundle, $B:F\to E$ a holomorphic epimorphism, and $p$ the induced metric on $E$:
\begin{equation}
p(v)=\inf\{q(w)\,:\, Bw=v\},\qquad v\in E,
\end{equation}
the curvatures of $p$ and $q$ are related by
\begin{equation}
K_\xi^p(v)\ge\inf\{K_\xi^q(w)\,:\, Bw=v\}.
\end{equation}
In particular, $K^p\ge 0$  if $K^q\ge 0$.
\end{lem}

\begin{proof}
The point is that in (2.5) the inf is attained. Indeed, if $s\in S$ and $v\in E_s$ is a nonzero vector, the kernel of $B|F_s$  being proximal (as any subspace of $F_s$, see \cite[V.4.6]{C}), $B^{-1}(v)$ contains  a shortest vector, that we will denote $w$. Given now a holomorphic map $f$ of a neighborhood of $0\in\bC$ into $S$, sending $\partial/\partial z\in T_0^{1,0}\bC$ to some  $\xi\in T_s^{1,0}S$, and a section $\psi$ of $F$, holomorphic near $s$, satisfying $\psi(s)=w$, let $\varphi=B\circ\psi$. Then 
$$
p(\varphi\circ f)\le q(\psi\circ f),
$$
with equality at $0\in\bC$. Hence
$$
\Lambda \log p(\varphi\circ f)\le\Lambda\log q(\psi\circ f)
$$
at $0$, and so $K_\xi^p(v)\ge K_\xi^q(w)$ by (2.2).
\end{proof}

We do not know if the lemma holds without the assumption of reflexivity.

For bundles of finite rank and for smooth metrics S. Kobayashi and Rochberg were the first to put forward the notions and results discussed above, see \cite{K, Ro}. In particular, Rochberg was the first to realize the connection between curvature and interpolation. 

\section{The main results}
The following results generalize Theorems 1.1 and 1.2:

\begin{thm} (i) Consider a holomorphic Hilbert bundle $E\to S$, endowed with a hermitian metric $h$ of class $C^2$, and a holomorphic Banach bundle $F\to S$ endowed with a metric $p$. If a holomorphic homomorphism $A:E\to F$ decreases curvature in the sense that 
$$K^p_\xi (Av)\le K^h_\xi(v) \text{ for all } s\in E_s, \xi\in T^{1,0}_s S, \text{ and } v\in E_s, Av\ne 0,
$$
then $\log\|A\|$is plurisubharmonic.

(ii) In particular, if $S$ is a half plane $S=\{s\in\bC: \text{Re  }s>\sigma_0\}, \|A\|$ is bounded, and $\|A\|(s)=\|A\|(\text{Re }s)$, then $\|A\|(t)$ is a decreasing function of $t\in(\sigma_0, \infty)$.
\end{thm}

\begin{proof}
Part (i) was proved, in a somewhat greater generality, in \cite[Theorem 2.4]{L}; see also the discussion at the end of that paper. When $\text{rk }E<\infty$, this was proved earlier in \cite{CS}, again in greater generality. Part (ii) immediately follows: under the assumptions $\log\|A\|(s)$ is subharmonic and independent of $\text{Im }s$, hence $\log\|A\|(t)$ is convex. If it is also bounded above, it must decrease.
\end{proof}

Somewhat related convexity theorems have already occurred in the framework of interpolation; the earliest seem to be in \cite{H, St}.

\begin{thm}
Consider a trivial Banach bundle $E=S\times V\to S$ over a half plane $S=\{s\in\bC: \text{Re }s>\sigma_0\}$, endowed with a metric $p$ such that $p_s$, viewed as norms on $V$, depend only on $\text{Re }s$. Assume that the second dual $p^{**}$ of $p$ has semipositive Kobayashi curvature $K^{**}\ge 0$. If for all $l\in V^*$, or at least for $l$ in a dense subset of $V^*$, 
\begin{equation}
\sup \big\{p^*_t (l):\quad t\in (\sigma_0,\infty)\big\}<\infty,
\end{equation}
then $p_t(v)$ is an increasing function of $t$, for all $v\in V$. In particular, $p_t(v)\le\lim\limits_{\tau\to\infty}p_{\tau}(v)$. 
\end{thm}

When $p$ is a Hilbertian metric, and more generally, when $V$ is reflexive, $(E^{**}, p^{**})$ and $(E, p)$ are isometrically isomorphic, so that the assumption $K^{**}\ge 0$ is the same as $K\ge 0$. But for a general $V$ we could not prove the theorem just assuming $K\ge 0$. 

\begin{proof}
By Lemma 2.1 the dual metric $p^*$ has seminegative curvature. Hence for any $l\in V^*$ the function $s\mapsto \log p^*_s(l)$ is subharmonic and $\log p^*_t(l)$ is convex in $t\in(\sigma_0,\infty)$. Knowing for a dense set of $l$'s that the latter function is bounded above implies it decreases, and then by continuity it in fact decreases for all $l\in V^*$.

Let now $v\in V$ and $t>\sigma_0$. By the Banach--Hahn theorem there is a nonzero $l\in V^*$ such that $\langle l, v\rangle=p^*_t(l)p_t(v)$. Thus
$$|\langle l, v\rangle|\le p^*_{\tau}(l) p_{\tau}(v)\quad \text{ for all }\tau\in(\sigma_0,\infty),$$
with equality when $\tau=t$. Since $p^*_{\tau}(l)\le p^*_t(l)$ when $\tau > t$, we must have $p_{\tau}(v)\ge p_t(v)$. This proves that $p_t(v)$ is an increasing function of $t$.
\end{proof}

\section{Application. Extending holomorphic sections}
We illustrate the technique of extrapolation, Theorem 3.2, by deriving an Ohsawa--Takegoshi type extension theorem. Ohsawa--Takegoshi type refers to extending from a submanifold $Y$ of a complex manifold $X$ holomorphic sections of certain vector bundles to all of $X$; the $L^2$ norm of the extension should be controled by the $L^2$ norm of the data, and the constant in this estimate should depend on crude geometric properties of $X$ and $Y$. Following the original work of Ohsawa and Takegoshi \cite{OT} a great many instances of such theorems were discovered---the list would be too long to reproduce here. In this section we will discuss a version and its proof that is a variant of what is done in \cite{BL}. The reader will be able to trace back to \cite{BL} most of the ideas in our current proof. If there is improvement it is that some of the ideas in \cite{BL} are no longer needed here. The advantage of this approach is that it is quite painless, and produces arguably sharp estimates. In the setting we treat here we recover the estimates that B\l ocki and Guan--Zhou obtain in \cite{Bl, GZ} by a different approach.

We will consider the problem to extend sections of the canonical bundle $K_X$ of an $m$ dimensional Stein manifold $X$. We could deal in the same way with bundles obtained by twisting $K_X$ by a Nakano semipositive vector bundle, but to keep notation simple we refrain from doing so. For sections $f$ of $K_X$, i.e., for $(m,0)$ forms, there is a natural $L^2$ norm $|\int_Xf\wedge\overline f|$; however, for sections of the restriction of $K_X$ to a submanifold $Y$ an $L^2$ norm, which will depend on the geometry of $Y$, has to be defined.

We start with an oriented smooth manifold $X$ of dimension $a<\infty$, and a submanifold $Y\subset X$ of dimension $b$. We assume given a function $r:X\to [0,\infty)$, of class $C^3$, whose zero set is $Y$, and whose critical points on $Y$ are transversely nondegenerate. Let $\cK\to X$ denote the bundle of real valued forms of (maximal) degree $a$. If $\omega\in\cK$, we denote by $|\omega|$ either $\omega$ or $-\omega$, whichever is a nonnegative multiple  of the orientation form. Given any compactly supported continuous section $\varphi$ of $\cK|Y$, we extend it to a compactly supported continuous section $\psi$ of $\cK$, and define an $L^1$--type norm of $\varphi$ by
\begin{equation}
\|\varphi\|_1=\lim_{\varepsilon\to 0}\varepsilon^{b-a}\int\limits_{\{x\in X: r(x)<\varepsilon^2\}}\, |\psi|.
\end{equation}
Near any point of $Y$ we can choose coordinates $x_1,\dots, x_a$ on $X$ so that $r=x^2_{b+1}+\cdots+x^2_a$. If $\varphi$ and $\psi$ are supported in this coordinate neighborhood $U$, and $\varphi=f dx_1\wedge\dots\wedge dx_a$, one easily computes that the limit in (4.1) exists and 
$$\|\varphi\|_1=\sigma \int_{Y\cap\, U} |f|\, dx_1\dots dx_b,$$
independently of the choice of $\psi$. Here $\sigma$ is the volume of the unit ball in $\bR^{a-b}$. By a partition of unity it follows that the limit in (4.1) always exists, and is independent of how we choose $\psi$.

Next, if $\varphi$ is a general continuous section of $\cK|Y$, we define $\|\varphi\|_1=\sup \|{\varphi}'\|_1\le\infty$, the sup taken over all compactly supported continuous sections ${\varphi}'$ such that $|{\varphi}'|\le|\varphi|$.

Let $\cK_Y\to Y$ denote the bundle of $b$--forms on $Y$. There is a pairing
$$\bigwedge^{a-b} TX|Y \times_Y \cK|Y\to \cK_Y,$$
given by contraction and restriction to $Y$. This pairing descends to a pairing of line bundles
$$\bigwedge^{a-b}N_Y\times_Y \cK|Y\to \cK_Y, \qquad (\xi,\varphi)\mapsto\iota_\xi\varphi$$
involving the normal bundle $N_Y=(TX|Y)/TY$ of $Y$. It is easy to see that if $Y$ is also oriented, $r$ determines a continuous section $\xi$ of $\bigwedge^{a-b}N_Y$ such that $\|\varphi\|_1$ can be computed as

$$\|\varphi\|_1=\int_Y|\iota_\xi\varphi|=\int_Y\iota_\xi|\varphi|.$$
For example, if in local coordinates $r=x^2_{b+1}+\cdots+x^2_a$, and $X,Y$ are oriented by $dx_1\wedge\dots\wedge dx_a$, respectively $dx_1\wedge\dots\wedge dx_b$, then $\xi$ will be (the class in $\bigwedge^{a-b}N_Y$ of) $(-1)^{a(a-b)}\sigma \partial / \partial x_{b+1}\wedge\dots\wedge\partial /\partial x_a$. We will not use this formula, but we will need an alternative representation of the limit in (4.1) in terms of a cut off milder than the sharp cut off $r<\varepsilon^2$. For this purpose we fix a continuous $\chi:\bR\to\bR$,
\begin{equation}
\chi(t)\begin{cases}
\quad =0,&\text{ for} \quad t\le 0\\
\quad >0,&\text{ for}\quad t>0,\end{cases}\qquad \liminf_{t\to\infty}\frac{\chi(t)}t>0.
\end{equation}

\begin{prop}
If $\psi$ is a compactly supported continuous section of $\cK$ and $\varphi$ its restriction to $Y$ (so a section of $\cK|Y$), then
\begin{equation}
\|\varphi\|_1=\lim_{t,T\to\infty} e^{(a-b)t/2} \int_X e^{-T\chi(t+\log r)}|\psi|.
\end{equation}
\end{prop}

\begin{proof}
Writing $X(t)=\{x\in X: r(x)<e^{-t}\}$,
$$e^{(a-b)t/2}\int_{X(t)} e^{-T\chi(t+\log r)} |\psi|=e^{(a-b)t/2} \int_{X(t)} |\psi|\to \|\varphi\|_1$$
as $t\to \infty$ by (4.1). To estimate the contribution of $X\setminus X(t)$ to the integral in (4.3) we can assume, as before, that $\psi$ is supported in a coordinate chart and in this chart $r=x^2_{b+1}+\cdots+x^2_a$. With a suitable $M\in \bR$ and by a change of coordinates $y=e^{t/2}x$
\begin{eqnarray*}
e^{(a-b)t\over 2}\negthickspace\negthickspace\int\limits_{X\setminus X(t)} \negthickspace e^{-T\chi(t+\log r)}|\psi| &\le& M e^{(a-b)t\over 2} \mspace{-18.0mu}\negthickspace\int\limits_{x^2_{b+1}+\cdots+x^2_a\ge e^{-t}} \mspace{-18.0mu}e^{-T\chi(t+\log(x^2_{b+1}+\dots+x^2_a))} dx_{b+1}\cdots dx_a\\
&\le& M\negthickspace\int\limits_{y^2_{b+1}+\cdots+y^2_a\ge 1}\mspace{-18.0mu} e^{-T\chi(\log(y^2_{b+1}+\cdots+y^2_a))} dy_{b+1} \cdots dy_a \to 0
\end{eqnarray*}
as $T\to\infty$ by dominated convergence.
 \end{proof}

Now let $X\supset Y$ be complex manifolds and $K_X$ the canonical  bundle of $X$. A function $r\in C^3(X)$ as above defines an $L^2$ norm for continuous sections $\varphi$ of $K_X|Y$,
\begin{equation}
\|\varphi\|^2_Y=\|i^{n^2}\varphi\wedge\overline\varphi\|_1\le \infty,\qquad n=\dim_\bC Y.
\end{equation}

\begin{thm}
Let $X$ be a Stein manifold and $Y\subset X$ a complex submanifold.
Suppose $r:X\to[0,1]$ is of class $C^3$, vanishes precisely on $Y$, and its critical points on $Y$ are nondegenenate in directions transverse to $Y$. If $\log r$ is plurisubharmonic, any holomorphic section $f$ of $K_X|Y$ can be extended to a holomorphic section $g$ of $K_X$ satisfying
\begin{equation}
\Big|\int_X g\wedge\overline g\Big| \le \|f\|^2_Y.
\end{equation}
\end{thm}

\begin{proof}
We will include this extension problem in a family of extension problems $P_{t,T}$ to which extrapolation can be applied and which becomes trivial in the limit $t,T\to \infty$. The extension problems will be of the same nature as what the theorem is claiming to solve, with the difference that on the left of (4.5) the $L^2$ norm will be replaced by weighted $L^2$ norms, with the weights concentrating near $Y$ more and more sharply as $t,T\to \infty$. Before constructing this family, though, we exhaust $X$ by a sequence $X_\nu$ of relatively compact pseudoconvex subsets. It will then suffice to find $g_\nu\in\cO(K_{X_\nu})$ extending $f|Y\cap X_\nu$ such that
$$
\Big|\int_{X_\nu}g_\nu\wedge\overline g_\nu\Big| \le \|f\|^2_Y+\frac1\nu,
$$
because once $g_\nu$ found, a subsequence will converge locally uniformly to a $g\in\cO(K_X)$ that extends $f$ and satisfies (4.5).

We pick a convex function $\chi:\bR\to\bR$ satisfying (4.2). With $T>0$ and $\nu$ fixed, for any $s\in\bC$ we define a Hilbert norm on a subspace $W\subset \cO(K_{X_\nu})$. If $\dim_\bC X=m$ and $\dim_\bC Y=n$, let
$$
q_s(\psi)=q_{s,T}(\psi)=e^{(m-n)\text{Re }s/2}\Big|\int_{X_\nu} e^{-T\chi(\text{Re  }s+\log r)}\psi\wedge\overline\psi\Big|^{1/2} \le \infty, \qquad \psi\in\cO(K_{X_\nu});
$$
the Bergman space $W=W_\nu$ consists of $\psi\in\cO(K_{X_\nu})$ for which $q_s(\psi)<\infty$ for some (and then for all) $s\in\bC, T>0$. On $W$ all the norms $q_s$ are equivalent, and
$$
q_s(\psi)=e^{(m-n)\text{Re}\, s/2}\Big|\int_{X_\nu}\psi\wedge\overline\psi\Big|^{1/2},\quad \text{ when } \text{Re }s\le 0.
$$
The norms $q_s$ together define a metric $q$ on the bundle $F=\bC\times W\to\bC$. Since the function
$$
\bC\times X_\nu \ni(s,x)\mapsto (n-m)\text{Re}\,s+T\chi(\text{Re }s +\log r(x))
$$
is plurisubharmonic, by Berndtsson's direct image theorem the Kobayashi curvature of $q$ is semipositive. \cite[Theorem 1.1]{B} deals with pseudoconvex subdomains of $\bC^m$ instead of Stein $X_\nu$, in which case the canonical bundle is trivial; but the proof carries over to Stein manifolds. Another difference is that \cite[Theorem 1.1]{B} claims Nakano semipositivity of the direct image bundle. However, for one dimensional bases like $\bC$ Nakano semipositivity is the same as Griffiths and Kobayashi semipositivity (the latter two are equivalent for arbitrary bases as long as the metrics involved are hermitian and of class $C^2$).

Consider next the closed subspace $I\subset W$ consisting of sections $\psi$ that vanish on $Y\cap X_\nu$, and the quotient $W/I=V=V_\nu$. The trivial bundle $\bC\times V\to \bC$ is a quotient of $F$ by the subbundle $\bC\times I\to \bC$, and inherits a Hilbertian metric $p$,
 $$
 p_s(v)=p_{s,T}(v)=\inf\{q_s(\psi)\, :\, \psi\in v\},\qquad v\in V=W/I.
 $$
In particular, $p_{s,T}$ is independent of $T$ when $\text{Re}\, s\le 0$. As a quotient of a semipositively curved metric, $p$ itself will be semipositively curved, see Lemma 2.3. Thus $K^{p**}=K^p\ge 0$. To apply Theorem 3.2 we need to check one more assumption, (3.1). We take an arbitrary $\sigma_0<0$, let $S=\{s\in\bC\, : \text{Re}\, s>\sigma_0\}$ and $E=S\times V\to S$.

So suppose $l:V\to\bC$ is a linear form. Composing it with the projection $W\to V$ we obtain a form $L\in W^*$ whose kernel contains $I$. Examples of such forms come from sections of $\overline K_{Y\cap X_\nu}\otimes T^{m-n, 0}X_\nu |Y$ as follows ($\overline K_{Y\cap X_\nu}$ stands for the bundle of $(0,n)$ forms on $Y\cap X_\nu$). Given $y\in Y$, $\lambda=\alpha\otimes\xi\in\overline K_Y|_y\otimes T^{m-n,0}_y X$, and $\psi\in K_X|_y$, set
$$
\lambda\barwedge\psi=\alpha\wedge\iota_\xi \psi.
$$ 
If now $\lambda$ is a compactly supported continuous section of $\overline K_{Y\cap X_\nu}\otimes T^{m-n, 0}X_\nu|Y$, we define $L_\lambda\in W^*$ by
$$L_\lambda(\psi)=\int_{Y\cap X_\nu}\lambda\barwedge\psi, \qquad \psi\in W.$$
Clearly, $L_\lambda$ vanishes on $I$, and linear forms of this type are dense among forms in $W^*$ that vanish on $I$, because $\psi\in W$ must be in $I$ if $L_\lambda(\psi)=0$ for all $\lambda$. In checking (3.1) we can restrict ourselves to $l\in V^*$ induced by such $L_\lambda$. Since $p^*_s(l)=q^*_s(L_\lambda)$, it suffices to check
$$
\sup_{s\in S} q^*_s(L_\lambda) <\infty
$$
for every compactly supported continuous $\lambda$. Even better, we may assume that $\lambda$ is supported in a coordinate patch on $X_\nu$, where $Y$ is given in local coordinates by $x_{n+1}=\dots=x_m=0$. We take a $\psi\in W$ and apply the submean value theorem to the coefficient in $\psi\wedge\overline\psi$ on balls
$$
x_1=\text{const},\dots, x_n =\text{const},\quad |x_{n+1}|^2+\cdots+|x_m|^2 <
\varepsilon e^{-\text{Re }s},
$$
with $\varepsilon>0$ small but fixed independently of $s\in S$. Letting 
$\xi=(\partial /\partial x_{n+1},\dots,\partial /\partial x_m)$, this gives 
$$
\Big |\int_{Y\cap\,\text{supp }\lambda}\iota_\xi\psi\wedge{\overline{\iota_\xi\psi}}\Big| \le C{q_s}(\psi)^2,
$$
$C$ independent of $\psi$ and $s\in S$. Hence $|L_\lambda(\psi)|\le C'{q_s}(\psi)$ holds, and so $\sup_{s\in S}p^*_s(l)=\sup_{s\in S}q_s^*(L_\lambda)<\infty$. Therefore Theorem 3.2 applies: for any $v\in V$ and $t>\sigma_0$
\begin{equation}
p_t(v)\le\lim_{\tau\to\infty} p_\tau(v).
\end{equation}

This implies the theorem as follows. By Cartan's Theorem B we can extend $f\in\cO(K_X|Y)$ to $\psi\in\cO(K_X)$. 
For each $\nu$ let $v_\nu\in V_\nu$ be the class of $\psi|X_\nu\in W_\nu$. Choose a compactly supported continuous function $\theta:X\to[0,1]$ that is 1 on $X_\nu$. By (4.6), Proposition 4.1, and (4.4)
\begin{align*}
\lim_{T\to \infty}p_{t,T}(v_\nu)^2\le&\lim\limits_{T\to\infty} \lim\limits_{\tau\to\infty}p_{\tau,T}(v_\nu)^2\le
\lim\limits_{T\to\infty}\limsup\limits_{\tau\to\infty} q_{\tau,T}(\psi|X_\nu)^2\\
\le&\lim\limits_{T\to\infty}\limsup\limits_{\tau\to\infty} e^{(m-n)\tau}\Big|\int_X e^{-T\chi(\tau+\log r)}\theta\psi\wedge\overline{\theta\psi}\Big|\le \|f\|^2_Y.
\end{align*}
This means that for every $t$ and $\varepsilon>0$, for sufficiently large $T$ there is a $g_\nu\in v_\nu$, i.e., $g_\nu\in W_\nu$ extending $f|Y\cap X_\nu$, such that $q_{t,T}(g_\nu)\le \|f\|_Y+\varepsilon$. Setting $t=0$ this implies
$$
\Big|\int_{X_\nu} g_\nu\wedge\overline g_\nu\Big|\le \|f\|^2_Y+\varepsilon,
$$
which, as we have seen, implies the theorem.
\end{proof}


\begin{thebibliography}{KMM}

\bibitem[B]{B} B.~Berndtsson, Curvature of vector bundles associated to holomorphic fibrations.
Ann.~of Math.~(2) 169 (2009) 531--560.
\bibitem[BL]{BL} B.~Berndtsson, L.~Lempert, A proof of the Ohsawa--Takegoshi theorem with sharp estimates.
arxiv:1407.4946.
\bibitem[Bl]{Bl} Z.~B\l ocki, Suita conjecture and the Ohsawa--Takegoshi extension theorem. 
Invent.~Math.~193 (2013) 149--158.
\bibitem[C]{C} J. B. Conway, A course in functional analysis, Graduate Texts in Mathematics, 2nd ed. 
1997, Springer, New York.
\bibitem[CS]{CS} R.~Coifman, S.~Semmes, Interpolation of Banach spaces, Perron processes, and Yang--Mills.
Amer.~J.~Math.~115 (1993) 243--278.
\bibitem[GZ]{GZ} Q.~Guan, X.~Zhou, A solution of an $L^2$ extension problem with an optimal estimate and applications. 
Ann.~of Math.~(2) 181 (2015) 1139--1208.
\bibitem[H]{H} I.~Hirschman, A convexity theorem for certain groups of transformations.
J.~Analyse Math.~(2) (1953) 209--218.
\bibitem[JM]{JM} B. Jawerth, M. Milman, Extrapolation theory with applications.
Memoirs of the Amer. Math. Soc. 89 (1991) No. 440 1--82.
\bibitem[K]{K} S. Kobayashi, Negative vector bundles and complex Finsler structures. Nagoya U. Math. J. 57 (1975) 153--166.
\bibitem[L]{L} L.~Lempert, A maximum principle for Hermitian (and other) metrics.
Proc. Amer. Math. Soc. 143 (2015) 2193--2200.
\bibitem[MM]{MM} J. Mart\'\i n, M. Milman, Extrapolation methods and Rubio de Francia's extrapolation theorem.
Adv. Math. 201 (2006)  209--262
\bibitem[OT]{OT} T.~Ohsawa, K.~Takegoshi, On the extension of $L^2$ holomorphic functions.
Math.~Z.~195 (1987) 197--204.
\bibitem[R]{R} M.~Riesz, Sur les maxima des formes bilin\'eaires et sur les fonctionnelles lin\'eaires.
Acta Math.~49 (1927) 465--497.
\bibitem[Ro]{Ro} R. Rochberg, Interpolation of Banach spaces and negatively curved vector bundles.
Pacific J. of Math. 110 (1984) 355--376.
\bibitem[Ru1]{Ru1} J.~Rubio~de~Francia, Vector valued inequalities for operators in $L^p$ spaces.
Bull. London Math. Soc., 12 (1980) 211--215.
\bibitem[Ru2]{Ru2} J.~Rubio~de~Francia, Factorization theory and $A_p$ weights.
Amer.~J.~Math.~106 (1984) 533--547.
\bibitem[Sa]{Sa} S.~Saks, On subharmonic functions.
Acta Litt. Sci. Szeged.~5 (1932) 187--193.
\bibitem[St]{St} E.~Stein, Interpolation of linear operators.
Trans.~Amer.~Math.~Soc., 83 (1956) 482--492.
\bibitem[T]{T} G.~Thorin, Convexity theorems generalizing those of M.~Riesz and Hadamard with some applications, Comm. Seri.~Math.~Univ.~Lund 9 (1948) 1-58.
\bibitem[Y]{Y} S. Yano, Notes on Fourier analysis, XXIX. An extrapolation theorem.
J. Math. Soc. Japan 3 (1951) 296--305
\end{thebibliography}
\end{document}